\documentclass[12pt]{amsart}
\usepackage[margin=1in]{geometry}
\usepackage[english]{babel}
\usepackage[utf8]{inputenc}
\usepackage{subcaption}
\usepackage{amsmath}
\usepackage{amssymb}
\usepackage{amsfonts}
\usepackage{amsthm}
\usepackage{mathdots}
\usepackage{mathrsfs}
\usepackage[all]{xy}
\usepackage[pdftex]{graphicx}
\usepackage{color}
\usepackage{cite}
\usepackage{url}
\usepackage{indent first}
\usepackage[labelfont=bf,labelsep=period,justification=raggedright]{caption}
\usepackage[english]{babel}
\usepackage[utf8]{inputenc}
\usepackage{hyperref}
\usepackage[colorinlistoftodos]{todonotes}
\usepackage{tkz-fct}
\usepackage{tikz}
\usetikzlibrary{calc}
\usepackage{pgfplots}
\usepackage{multicol}
\PassOptionsToPackage{dvipsnames,svgnames}{xcolor}
\usepackage{textcomp}
\usepackage{bbm}
\usepackage[ruled,vlined]{algorithm2e}
\usepackage[shortlabels]{enumitem}

\newcommand{\NN}{\mathbb{N}}

\newcommand{\ZZ}{\mathbb{Z}}

\newcommand{\T}{\mathcal{T}}

\theoremstyle{plain}
\newtheorem{thm}{Theorem}

\newtheorem{prop}[thm]{Proposition}

\theoremstyle{definition}
\newtheorem{definition}[thm]{Definition}

\theoremstyle{remark}

\allowdisplaybreaks

\numberwithin{equation}{section}
\numberwithin{thm}{section}

\begin{document}

\title{Hypergraph Fuss-Catalan Numbers}
\author{Parth Chavan}\author{Andrew Lee} \author{Karthik Seetharaman}
 \address{Euler Circle, Palo Alto, CA}
 \email{spc2005@outlook.com}
\address{Massachusetts Academy of Math and Science At WPI, MA}
\email{leeandrew1029@gmail.com}
\address{Massachusetts Academy of Math and Science At WPI, MA}
\email{kvseetharaman2@gmail.com}
\maketitle

\begin{abstract}
The Catalan numbers $C_n$ are an extremely well-studied sequence of numbers that appear as the answer to many combinatorial problems. Two generalizations of these numbers that have been studied are the Fuss-Catalan numbers and the Hypergraph Catalan numbers. In this paper, we study the combination of these, the Hypergraph Fuss-Catalan numbers. We provide some combinatorial interpretations of these numbers, as well as describe their generating function.
\end{abstract}

\section{Introduction}

The $\emph{Catalan numbers}$ $(C_n)_{n \geqslant 0}$ 
$$1,1,2,5,14,42,132,429,1430,\ldots$$ are one of the most ubiquitous sequences in mathematics. They have several combinatorial interpretations, some of which we include in Section~\ref{catalan}. The comprehensive source for combinatorial interpretations of the Catalan numbers is  \cite{stanley2015catalan}, which lists no less than 214 interpretations of these numbers. Given some $r \in \mathbb{N}$, most of these combinatorial interpretations readily generalize to a sequence $(FC_{n}^{(r)})_{n \geqslant 0}$ known as the \textit{Fuss-Catalan numbers}. For example,  \cite{bergeron2012combinatorics} defines these numbers by extending the notion of a Dyck path to a $r$-Dyck path. 

The main purpose of this paper is to further generalize the Fuss-Catalan numbers. We do this by combining the Fuss-Catalan numbers with another generalization of the Catalan number known as the \textit{Hypergraph Catalan numbers} \cite{MR4245285}.  For a given $m \in \NN$, the sequence of numbers $\{C_n^{(m)}\}_{n \geqslant 0}$ denotes the associated sequence of Hypergraph Catalan numbers. We more rigorously define these numbers in Section~\ref{hypergraph}. 

Next, we generalize the Hypergraph Catalan numbers to the \textit{Hypergraph Fuss-Catalan numbers} by adding the $r$ parameter to the Hypergraph Catalan number. 

In particular, for each positive integer $r,m$, we define a sequence of integers $\{FC_{n}^{(r,m)}\}_{n \geqslant 0}$, which is the associated sequence of \textit{Hypergraph Fuss-Catalan numbers}. These are formally defined in Definition \ref{def2.7}. When $m=1$, the sequence $\{FC_{n}^{(r,1)}\}_{n \geqslant 0}$ is equivalent to the sequence of Fuss-Catalan numbers $\{FC_{n}^{(r)}\}_{n \geqslant 0}$, and when $r=1$, the sequence $\{FC_{n}^{(1,m)}\}_{n \geqslant 0}$ is equivalent to the sequence of Hypergraph Catalan numbers $\{C_{n}^{(m)}\}_{n \geqslant 0}$. Below are more examples of Hypergraph Fuss-Catalan numbers for parameters $r,m > 1$ (parameters are kept small as the sequence of numbers tends to grow extremely fast).
\begin{align*}
& r=2,m=2: 1,1,144,1341648,693520980336,\ldots \\
& r=2,m=3: 1,1,480,200225,18527520,45589896150400,\ldots\\ 
& r=3,m=2: 1,1,11532,628958939250,163980917165716725552156,\ldots\\ 
& r=3,m=3: 1,1,38440,8272793255000,9396808005460764741084000,\ldots 
\end{align*}

The family of numbers $(FC_{n}^{(r,m)})_{n \geqslant 0}$ are defined in terms of counting walks on hypertrees, weighted by the order of their automorphism groups. This definition is created by generalizing the corresponding definition for the Catalan numbers. From our definition, it is not trivial that $(FC_{n}^{(r,m)})_{n \geqslant 0}$ are integers, but we prove this by giving several equivalent combinatorial interpretations of the sequence. 
In Section~\ref{pre}, we include basic definitions of Catalan numbers, Fuss-Catalan numbers, and Hypergraph Fuss-Catalan numbers. In Section~\ref{Co}, we include five interpretations of $(FC_{n}^{(r,m)})_{n \geqslant 0}$, which generalize five standard interpretations of $(C_{n})_{n \geqslant 0}$. Moreover, we show that all of these interpretations actually count the Hypergraph Fuss-Catalan numbers through a series of bijections. In Section~\ref{Genfunc}, we explain how to compute the generating function of $(FC_{n}^{(r,m)})_{n \geqslant 0}$, which allows for rapid computation of the Hypergraph Fuss-Catalan Numbers. 
\section{Preliminaries}\label{pre}

\subsection{Catalan Numbers}\label{catalan}

The \textit{Catalan Numbers} $\{C_n\}_{n \geqslant 0}$ are a very well-studied sequence of numbers that show up as the answer to many combinatorial problems. They are given by the explicit formula $$C_n = \frac{1}{n+1}{2n \choose n}$$ for all $n \geqslant 0$ and have generating function $$c\left(x\right) = \dfrac{1-\sqrt{1-4x}}{2}.$$ Five standard combinatorial interpretations of the Catalan numbers are given below as in \cite{stanley2015catalan}; these five interpretations will be of particular interest to us later in the paper. 

\begin{enumerate}
    \item The number of rooted plane trees on $n$ edges is $C_n$. 
    \item The number of binary trees with $n$ internal nodes is $C_n$.
    \item A \textit{Dyck path} is a path in the upper half of the Cartesian plane from $(0,0)$ to $(2n,0)$ such that each step either goes from $(x,y)$ to $(x+1,y+1)$ or $(x+1,y-1)$. The number of Dyck paths for a given integer $n$ is $C_n$.
    \item A \textit{ballot sequence} is a sequence of $1$s and $-1$s such that every partial sum of the sequence is nonnegative and the whole sequence sums to $0$. There are $C_n$ ballot sequences of length $2n$.
      \item $C_n$ is the number of ways to divide a regular $(n+2)$-gon into triangles without adding new vertices and by drawing $n - 1$ new diagonals.
\end{enumerate}

\begin{figure}
    \centering
    \includegraphics{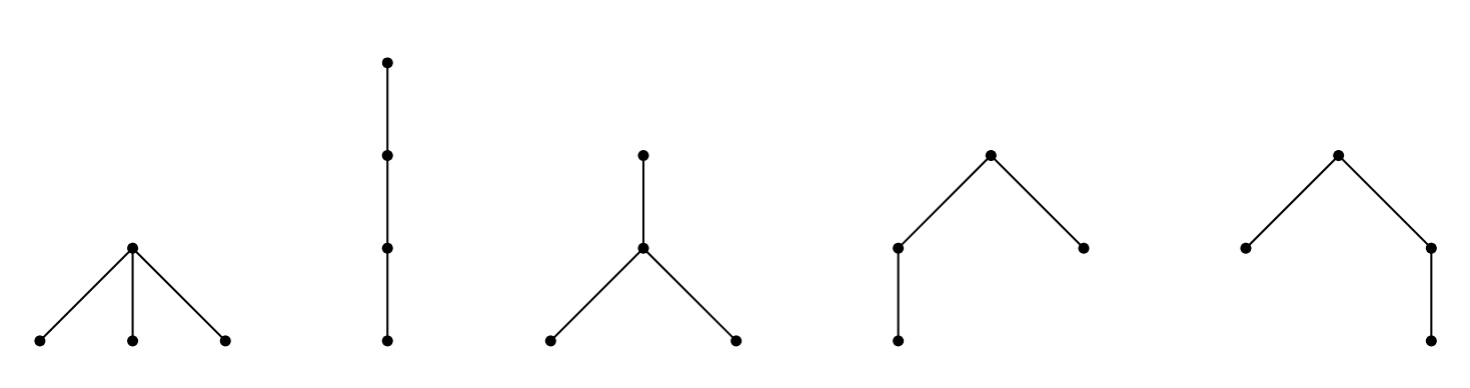}
    \caption{The 5 rooted plane trees with 3 edges.}
    \label{fig:rooted plane trees}
\end{figure}



We quickly note bijections between these interpretations to show that they all count the same sequence of numbers. In particular, a Dyck path can be obtained from a rooted plane tree by performing a preorder traversal of the tree (visits the root, then recursively visits each subtree defined by a child of the root until every vertex is visited). Consider constructing a Dyck path while performing the traversal. Every time an edge is traveled that goes further from the root, a segment of the form $(x,y) \rightarrow (x+1,y+1)$ is added to the Dyck path. Every time an edge is traveled that goes closer to the root, a segment of the form $(x,y) \rightarrow (x+1,y-1)$ is added to the Dyck path. The obtained path always stays at or above the $x$-axis because the $y$-coordinate of a point on the Dyck path corresponds to the distance from the root at that point in the traversal, which is always nonnegative. This process is reversible, which yields a bijection. 

To give a bijection between Dyck paths and ballot sequences, we note that a segment from $(x,y)$ to $(x+1,y+1)$ in a Dyck path corresponds to 1 in a ballot sequence, and a segment from $(x,y)$ to $(x+1,y-1)$ in a Dyck path corresponds to $-1$ in a ballot sequence. Although these two bijections are well-documented in the literature, we provide them here to make their eventual generalizations clearer.

\subsection{Hypergraph Catalan Numbers}\label{hypergraph}

The \textit{Hypergraph Catalan numbers} are another important generalization of Catalan numbers defined based on yet another interpretation, which we now provide. For a positive integer $n$, let $\T_n$ be the set of unlabeled trees on $n$ vertices. For a tree $T \in \T_n$ and vertex $v \in T$, define $a_T(v)$ to be the number of walks that:

\begin{itemize}
    \item Start and end at $v$. 
    \item Travel each edge twice, once going away from $v$ and once coming back from $v$ (in the sense of the tree being rooted at $v$).
\end{itemize}
Let $\Gamma(T)$ denote the automorphism group of $T$ and $|\Gamma(T)|$ be its order.
The relation between walks on trees and catalan numbers is given by following result:
\begin{prop}
The Catalan number $C_{n}$ is given by
$$C_{n} = \sum_{T \in \T_{n+1}} \sum_{v \in T} \frac{a_T(v)}{|\Gamma(T)|},$$ .

\end{prop}

We do not provide the proof here, but instead refer the reader to \cite{MR4245285}. Figure~\ref{fig:c5} shows the calculation of $C_5$ in this manner. 

\begin{figure}
    \centering
    \includegraphics{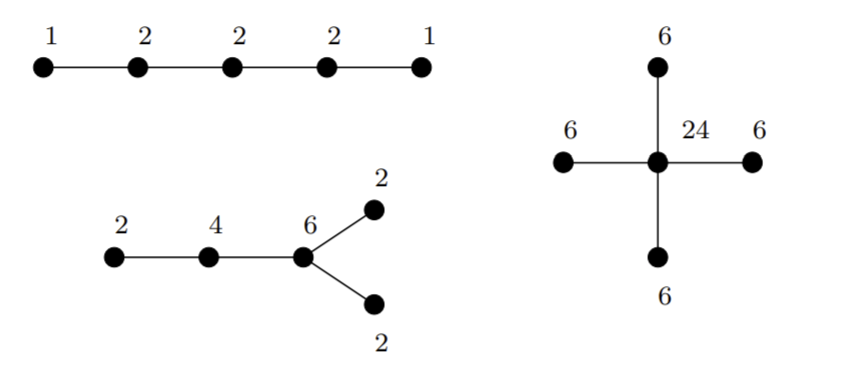}
    \caption{\cite{gunnells2021handout}  Calculating $C_5$ by walking on trees. Each vertex of each tree is labeled with the value of $a_T(v)$.}
    \label{fig:c5}
\end{figure}

The Hypergraph Catalan numbers generalize the above formula but with an additional parameter $m$. More specifically, for a tree $T \in \T_n$ and vertex $v \in T$, define $a_T^{(m)}(v)$ to be the number of walks that:

\begin{enumerate}
    \item Start and end at $v$.
    \item Travel each edge $2m$ times, $m$ times going away from $v$ and $m$ times coming back. 
\end{enumerate}
\begin{definition}
 
The Hypergraph Catalan numbers $C_{n}^{(m)}$ are defined by
 $$C_{n}^{(m)} = \sum_{T \in \T_{n+1}} \sum_{v \in T} \frac{a_T^{(m)}(v)}{|\Gamma(T)|}.$$  
 \end{definition}
 Information on the generating function as well as combinatorial interpretations of these numbers are included in \cite{MR4245285}. 

If $m=2$, the first few Hypergraph Catalan numbers are $$C_0^{(2)} = 1, C_1^{(2)}=1, C_2^{(2)}=6, C_2^{(3)}=57, C_2^{(4)}=678,$$ and so on. The calculation for $C_2^{(4)}$ using the tree-walking method is shown in Figure~\ref{fig:hc4}. 

\begin{figure}
    \centering
    \includegraphics{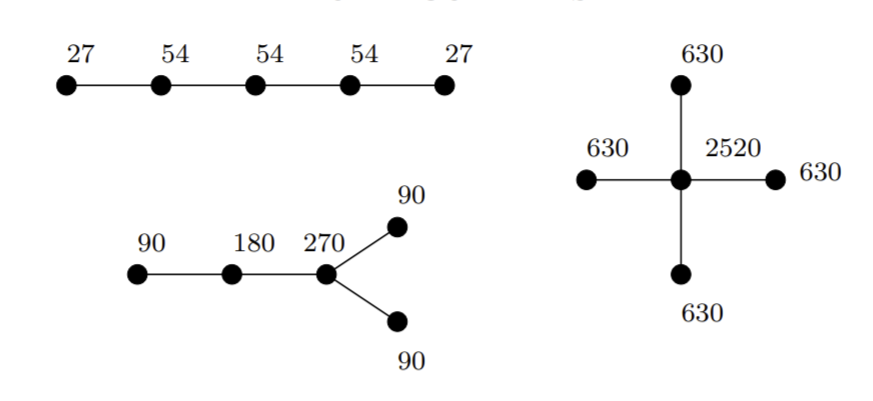}
    \caption{\cite{gunnells2021handout}  Calculating $C_4^{(2)}$ by walking on trees. Each vertex of each tree is labeled with the value of $a_T^{(2)}(v)$.}
    \label{fig:hc4}
\end{figure}

\subsection{Fuss-Catalan Numbers}\label{fuss1}

The \textit{Fuss-Catalan numbers} are a natural generalization of the Catalan numbers. Each positive integer $r$ yields a sequence of Fuss-Catalan numbers $FC_n^{(r)}$ for $n \geqslant 0$. In particular, we have $$FC_n^{(r)} = \frac{1}{rn+1}{\left(r+1\right)n \choose n}$$ for all $n \geqslant 0$. Note that $r=1$ corresponds to the standard Catalan numbers. 

Furthur examples are as follows
\begin{align*}
& r=2: \,\,\, 1, 1, 3, 12, 55, 273, 1428, 7752, 43263, 246675 , \ldots \\
& r=3: \,\,\, 	1, 1, 4, 22, 140, 969, 7084, 53820, 420732, 3362260, \ldots \\
& r=4: \,\,\, 1, 1, 5, 35, 285, 2530, 23751, 231880, 2330445, 23950355 , \ldots
\end{align*}
which can be found in OEIS $A001764$, $A002293$ and $A002294$. A \textit{hypergraph} is a generalization of a graph such that edges are non-empty subset of the set of vertices, which are called \emph{hyperedges}. Formally we define a hypergraph as given in \cite{berge1973graphs} .

 \begin{definition}
     Let $X = \{x_1,x_2, \ldots,x_n\}$ be a finite set and let $E = \{E_i | i \in I\}$ be a finite family of subsets of $X$. The pair $(X,E)$ is said to be a \emph{hypergraph} if 
\begin{enumerate}
    \item $E_i \neq \phi \, \forall i \in I$
    \item $\cup_{i \in I} E_i = X$
\end{enumerate}
 $E$ is called the set of \emph{hyperedges}. Call a \emph{hypergraph}  \emph{$k$-uniform} if $|E_i|=k$ for all $i \in I$.
 \end{definition}
 We will be working with uniform hypertrees throughout the paper which are defined below.
 \begin{definition}
    A $k$-uniform \emph{hypertree} on $n$ \emph{hyperedges} is a $k$-uniform \emph{hypergraph} on $n$ \emph{hyperedges} with $n(k-1)+1$ vertices.
    \end{definition}
    A \emph{rooted uniform hypertree} is a \emph{uniform} hypertree with a distinguished hyperedge called \emph{root hyperedge}. We call one of the lefmost or topmost vertex from the \emph{root hyperedge} as \emph{root vertex}. We call a \emph{rooted uniform hypertree} a $(r+1)$\emph{uniform rooted plane hypertree} if it contains a single hyperedge $e$ with $r+1$ vertices or else it has a subsequence $(P_1,P_2,\ldots,P_{r+1})$ of $(r+1)$uniform rooted plane hypertrees $P_i$,\, $1\leq i \leq r+1$. Thus the subtrees attached to the root hyperedge are linearly ordered. When drawing such trees the root $e$ is written on top with root vertex of $P_i$
    overlapping the $i^{\text{th}}$ of $e$ relative to the \emph{root vertex} of $e$. For instance \ref{fig:12hypertrees} shows two hypertrees $X$ and $Y$ with the horizontal hyperedge the root hyperedge and leftmost vertex of horizontal hyperedge the root vertex. $Y$ is not a $3$-uniform rooted plane hypertree because root vertex of the vertical hyperedge is the topmost vertex which is not attached to the root hyperedge.
    
    Fuss-Catalan numbers also count uniform rooted plane hypertrees as established by the following Theorem. We didn't find this interpretation in literature. Figure \ref{fig:12hypertrees} illustrates the $12$, $3$-uniform hypertrees on $3$ edges and it can be checked from the examples that $FC_{3}^{(2)} = 12$. 
    \begin{figure}
    \centering
    \includegraphics[width=5cm,height=3.3cm,keepaspectratio]{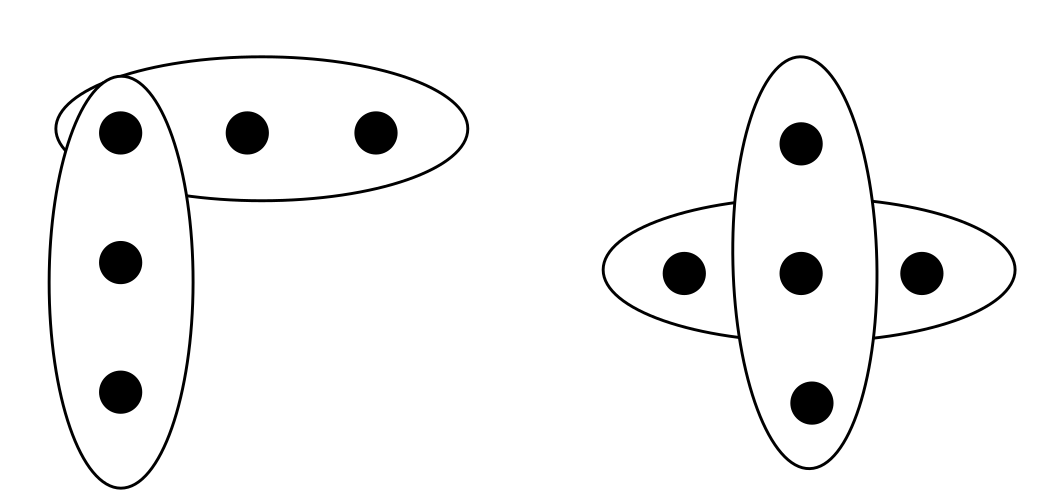}
    \caption{Two $3$-uniform hypertrees $X$ and $Y$. $Y$ is not a $3$-uniform rooted plane hypertree.}
    \label{fig:hypertree}
\end{figure}
\begin{figure}
    \centering
    \includegraphics[width=14cm,height=7cm,keepaspectratio]{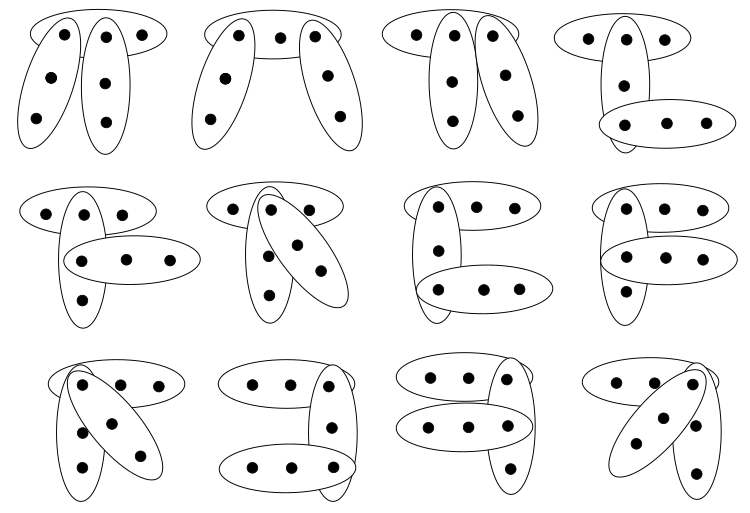}
    \caption{$3$-uniform rooted plane hypertrees on $3$ edges where the topmost horizontal hyperedge is the root hyperedge for each hypertree. }
    \label{fig:12hypertrees}
\end{figure}
\begin{thm}

The number of $\left(r+1\right)$ uniform rooted plane hypertrees on $n$ hyperedges with ordering specified for children is $FC_{n}^{(r)}$
\end{thm}
\begin{proof} Let $\mathcal{T}_{\left(r+1\right)}$ be the class of $\left(r+1\right)$ uniform plane hypertrees with. Fix an $\left(r+1\right)$ uniform hyperedge which will be the root hyperedge. Then any hypertree can be specified by a sequence  $(T_1,T_2,\ldots,T_{r+1})$ of rooted plane hypertrees which are dangling from Left to right relative to the root vertex in the root hyperedge from definition. This shows that the generating function $T_{(r+1)}(x)$ of $(r+1)$-uniform rooted plane hypertrees satisfies $T_{\left(r+1\right)}\left(x\right)=1+xT^{\left(r+1\right)}\left(x\right)^{r+1}$. The function $T'_{\left(r+1\right)}\left(x\right)=T'_{r+1}\left(x\right)-1$ satisfies $T'_{r+1}\left(x\right)=x\left(1+T'_{r+1}\left(x\right)\right)^{r+1}$. Thus using Lagrange inversion
$$\left[x^n\right]T'_{\left(r+1\right)}\left(x\right)=\frac{1}{n}\left[z^{n-1}\right]\left(1+z\right)^{n\left(r+1\right)}=\frac{1}{n}\binom{n\left(r+1\right)}{n-1}=\frac{1}{rn+1}\binom{n\left(r+1\right)}{n}$$
as desired.
\end{proof}

We provide interpretations of the Fuss Catalan numbers which are generalizations of the combinatorial interpretations provided for Catalan numbers. These can be found \cite{hilton1991catalan}.
\begin{enumerate}
    \item The number of rooted $\left(r+1\right)$-uniform plane hypertrees with $n$ hyperedges with ordering specifies for children. (An $\left(r+1\right)$-uniform hypertree is a generalization of a tree, where each edge is a set of $r+1$ distinct vertices and the hypertree has $rn+1$ vertices, if it has $n$ hyperedges).
    \item The number of incomplete $\left(r+1\right)$-ary trees with $n+1$ vertices .
    \item We generalize the Dyck path to an $r$-Dyck path, which is a path from $(0,0)$ to $(\left(r+1\right)n,0)$ such that each step either goes from $(x,y)$ to $(x+1,y+1)$ or to $(x+1,y-r)$, with the path at or above the $x$-axis at all times. The number of $r$-Dyck paths for some $n$ is $FC_n^{(r)}$. 
    \item We generalize a ballot sequence to an $r$-ballot sequence, which is a sequence of $1$ and $-r$ such that every partial sum of the sequence is nonnegative and the whole sequence sums to 0. There are $FC_n^{(r)}$ $r$-ballot sequences in which $1$ appears $rn$ times and $-r$ appears $n$ times.
    \item The number of ways to divide a regular $(nr+2)$-gon into $(r+2)$-gons without adding new vertices and by drawing $n-1$ new diagonals
\end{enumerate}

We can also define the Fuss-Catalan numbers in a similar manner to the definition of the Catalan numbers given in Section~\ref{hypergraph}. Specifically, instead of considering trees on $n$ vertices, we consider $\left(r+1\right)$-uniform rooted plane hypertrees with $n$ hyperedges. Let $\mathcal{H}_n^{(r)}$ be the set of unlabeled $r$-uniform rooted plane hypertrees with $n$ hyperedges. 

Given any $T \in \mathcal{H}_n^{(r)}$, we can define an associated bipartite graph $\Tilde{T}$ that is often easier to work with. Specifically, represent each vertex in $T$ by a black vertex in $\Tilde{T}$, and each hyperedge in $T$ by a white vertex in $\Tilde{T}$. The black and white vertices form the two groups of vertices in the bipartite graph. A black vertex is connected to a white vertex in $\Tilde{T}$ exactly when the corresponding vertex in $T$ is contained in the hyperedge represented by the white vertex. This implies each white vertex in $\Tilde{T}$ has degree $r+1$. A hypertree $T$ and its associated $\Tilde{T}$ are shown in Figure~\ref{fig:association}. Note that this means a walk on a hypertree $T \in \mathcal{H}_n^{(r)}$ corresponds to a walk on $\Tilde{T}$ that starts on a black vertex. However, we also require there to be no loops in our walk. This corresponds to no subsequences of the form $v_1v_2v_1$ in our walk on $\Tilde{T}$, where $v_1$ is black and $v_2$ is white. 

\begin{figure}
    \centering
    \includegraphics[width=12cm,height=5cm,keepaspectratio]{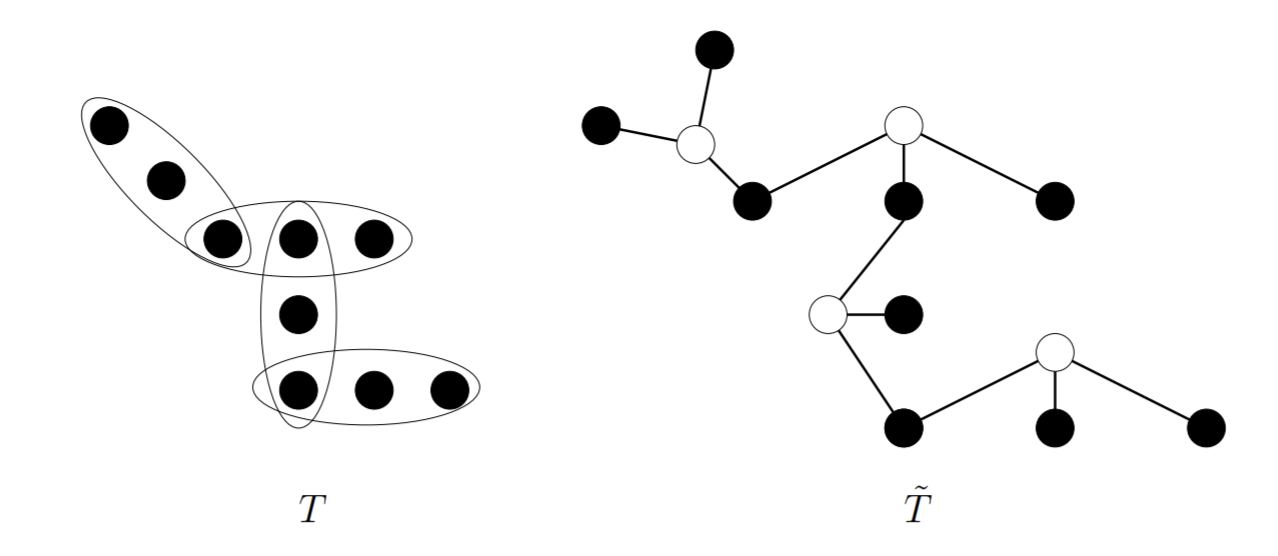}
    \caption{\cite{gunnells2021handout} A hypertree $T$ and its associated $\Tilde{T}$.}
    \label{fig:association}
\end{figure}

Take a $T \in \mathcal{H}_n^{\left(r+1\right)}$ and consider its associated $\Tilde{T}$. For any black vertex $v \in \Tilde{T}$, define $a_T^{(r)}(v)$ to be the number of walks on $\Tilde{T}$ that:

\begin{enumerate}
    \item Begin and end at $v$.
    \item Visit every black vertex of $\Tilde{T}$. 
    \item Use every edge of $\Tilde{T}$ twice, once going away from $v$ and once coming back to $v$. 
    \item There exist no subsequences of the form $v_1v_2v_1$, where $v_1$ is black and $v_2$ is white. 
\end{enumerate}

The relation between walks on Hypergraphs and Fuss Catalan numbers is given by the following result:
\begin{prop}
The Fuss Catalan number $FC_{n}^{(r)}$ is given by
$$FC_n^{(r)} = \sum_{T \in \mathcal{H}_n^{\left(r+1\right)}} \sum_{v \in \Tilde{T}} \frac{a_T^{(r)}(v)}{|\Gamma(T)|},$$ 
where the inner sum is taken over black vertices of $\Tilde{T}$
\end{prop}
 An example of this computation is shown in Figure~\ref{fig:fuss2}. It is immediately not clear how $FC_{n}^{r}$ are integers from the proposition. However we prove it in Theorem \ref{3.2} by giving their combinatorial interpretations.
\begin{figure}
    \centering
    \includegraphics[width=14cm,height=7cm,keepaspectratio]{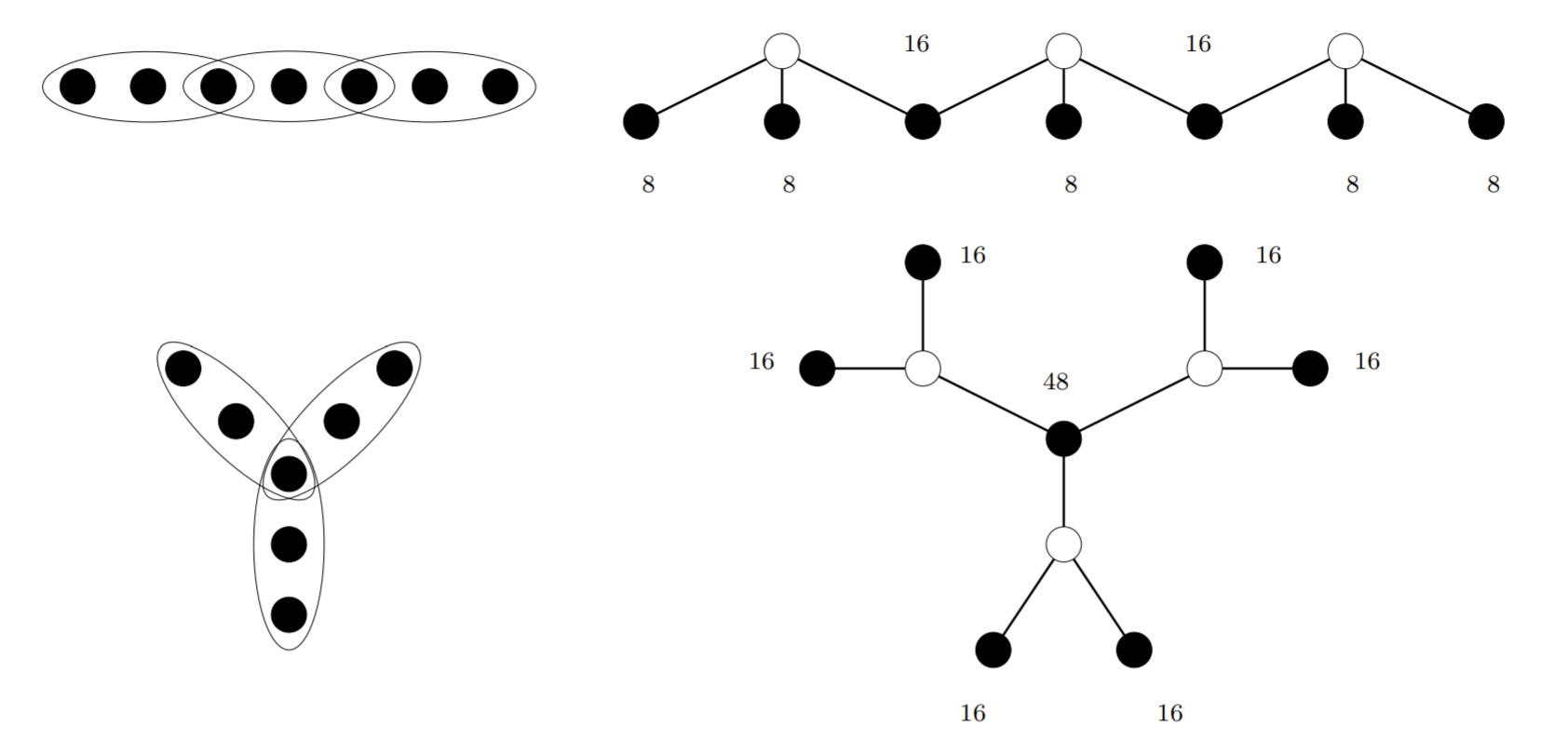}
    \caption{\cite{gunnells2021handout} Calculating $FC_3^{(2)}$ using the tree walking method.}
    \label{fig:fuss2}
\end{figure}
\subsection{Hypergraph Fuss-Catalan Numbers}\label{hypergraph fuss}

The main objects of interest in this paper are the \textit{Hypergraph Fuss-Catalan numbers}. These are a combination of the Hypergraph Catalan numbers and Fuss-Catalan numbers. To define them, we generalize the walks on hypertrees as given in definition of $FC_n^{(r)}$ in a similar way to the way Hypergraph Catalan numbers generalize Catalan numbers.

Given $T \in \mathcal{H}_n^{\left(r+1\right)}$, we consider its associated $\Tilde{T}$, and for any black vertex $v \in \Tilde{T}$, we define $a_T^{\left(r,m\right)}(v)$ to be the number of walks on $\Tilde{T}$ that:

\begin{enumerate}
    \item Begin and end at $v$.
    \item Use every edge of $\Tilde{T}$ exactly $2m$ times, $m$ times going away from $v$ and $m$ times coming back to $v$. 
    \item There exist no subsequences of the form $v_1v_2v_1$, where $v_1$ is black and $v_2$ is white. 
\end{enumerate}
\begin{definition}\label{def2.7}
 Let $T \in \mathcal{H}_n^{\left(r+1\right)}$ and $\Tilde{T}$ be its associated graph. Define  $$FC_n^{\left(r,m\right)} = \sum_{T \in \mathcal{H}_n^{\left(r+1\right)}} \sum_{v \in \Tilde{T}} \frac{a_T^{\left(r,m\right)}(v)}{|\Gamma(T)|},$$ where the inner sum is taken over the black vertices of $\Tilde{T}$. 
\end{definition}

\section{Combinatorial Interpretations}\label{Co}

As in Catalan numbers and the Fuss-Catalan numbers, the hypergraph Fuss-Catalan numbers $FC_{n}^{\left(r,m\right)}$ have similar combinatorial interpretations in terms of objects used to count $FC_{mn}^{\left(r\right)}$. We use Gunnell's approach in \cite{MR4245285}  for regular Hypergraph Catalan numbers and generalize it to work for Hypergraph Fuss-Catalan numbers. \\
Let $X$ be a combinatorial object. The reader should imagine of $X$ being some standard interpretation of Fuss-Catalan numbers as given in \ref{fuss1}. We give examples in \ref{3.1}. $X$ will usually be a set of smaller elements $x$, and we say that level structure for $X$ is a surjective map $l$ from these elements $x$ to a finite set $\left[N\right]=\{1,2,\ldots,N\}$ for some $N \in \NN$ of labels. We say that $x \in X$ is on a higher level than $x' \in X$ if $l\left(x\right)>l\left(x'\right)$ with similar convention for same and lower level. The $i^{\text{th}}$ level $X_i \subset X$ with respect to $l$ will be $l^{-1}(i)$. Note that $l$ will typically correspond to height of elements in $X$.
\\
We say that $x$ is \emph{parent} of \emph{child} $x'$ if $l\left(x'\right)=l\left(x+1\right)$. We will consider $m$-labelling of levels of $X$. First we fix an infinite set $L$ of labels. Next, let \[X = X_0 \bigsqcup_{i \geqslant 1} X_i\] be disjoint union of levels. For each level $X_i$ we choose a set partition into subsets of order $m$; this implies $\left|X_i\right| \equiv 0 \pmod m$. We say that a labelling is admissible if the following conditions are satisfied.
\begin{enumerate}
    \item Distinct subsets receive distinct labels
    \item Elements in same partition are in same class(which we define later)
    \item If two elements $x,x'$ share the same label then the labels of their parents agree
\end{enumerate}
We also consider two labellings equivalent if one is obtained from other by permuting the labels.
In each of our combinatorial interpretations, we define a set of objects $\mathscr{X}_{mn}^{\left(r\right)}$ constituting combinatorial interpretations of $FC_{n}^{\left(r\right)}$. Each of these objects consist of elements, for which we will define which elements have a parent/child relation. 

\begin{figure}
    \centering
    \includegraphics[width=7.4cm,height=4.5cm,keepaspectratio]{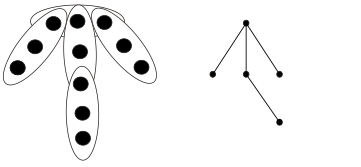}
    \caption{A $3$ uniform hypertree and corresponding ternary tree where the directions are left, middle and right.}
    \label{fig:fuss3}
\end{figure}
\begin{figure}
    \centering
    \includegraphics[width=8cm,height=4cm,keepaspectratio]{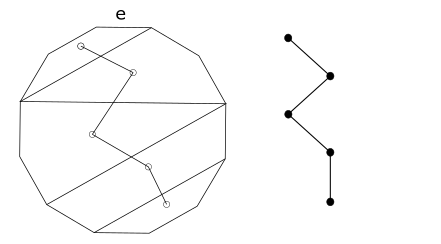}
    \caption{Polygonal division and corresponding ternary tree}
    \label{fig:fuss4}
\end{figure}

Next, we define level structure on $\mathscr{X}_{mn}^{\left(r\right)}$ and prove that
\[FC_{n}^{\left(r,m\right)} = \left(FC_{1}^{\left(r,m\right)}\right)^n \sum_{X \in \mathscr{X}_{mn}^{\left(r\right)}} N_m(X),\]
where $N_m(X)$ denotes admissible $m$-labellings of $X$.

Next, we present several combinatorial interpretations 
of $\mathscr{X}_{mn}^{\left(r\right)}$.
\begin{enumerate}
    \item \textbf{Plane Hypertrees.}
    The set $\mathscr{X}_{1,mn}^{\left(r\right)}$ is the set of $(r+1)$-uniform rooted plane hypertrees on $rnm+1$ vertices. The elements of a plane hypertree are its hyperedges. Moreover, the level of an hyperedge is its maximum distance to the root from any of its vertices. Two edges are of the same class if they both contain the root and  the vertices they share with their parents correspond to each other in respect to the parents.
    \item \textbf{Dyck Paths.}
    The set $\mathscr{X}_{2,mn}^{\left(r\right)}$ is the set of $r$-Dyck paths of length $mn$ (paths from $(0,0)$ to $((r+1)mn,0)$ where each step is of the form $(x,y) \to (x+1,y+r)$ or $(x,y) \to (x+1,y-1)$ and the path always stays above the $x$-axis), where the elements of a path are its slabs. We define a slab to be connected component that is bounded by the Dyck path, $y = a$ and $y = a + r$ for some nonnegative integer $a$ such that its left bound is an $(x,y) \to (x+1,y+r)$ step in the Dyck path. A slab $S$ is the parent of slab $S'$ if the the bottom horizontal edge of $S'$ is contained in the top horizontal edge of $S$. Moreover, two slabs are of the same class if and only if they have the same congruency class modulo $r$.
    
    \item \textbf{Ballot Sequences.}
    The set $\mathscr{X}_{3,mn}^{\left(r\right)}$ is the set of ballot sequences $$B = \left(a_1, a_2, \ldots, a_{mn\left(r+1\right)}\right),$$ where $a_i = r$ or $-1$ and $\sum_1^{mn\left(r+1\right)} a_i = 0$. Let $s_k = \sum_1^k a_i$. We say $(i,j)$, $i < j$, is a pair if (i) $a_i = r, a_j = -1$, (ii) $s_i = s_j + r$, and (iii) $j$ is the minimal index greater than $i$ for which the conditions hold true. The elements of a ballot sequence are its pairs. The pair $(i,j)$ is the parent of of pair $(k,l)$ if $s_i + k = s_k$, $i < k$, and $j > l$. Moreover, two pairs, $(i, j), (k,l)$ are of the same class if and only if $s_i = s_k \pmod r$.
    
    \item \textbf{$(r+1)$-ary Trees.}
    The set $\mathscr{X}_{4,mn}^{\left(r\right)}$ is the set of $(r+1)$-ary trees with $mn$ vertices. In an $(r+1)$-ary tree, each vertex has $r+1$ possible positions directly below it. We say an edge of a path is a left step if the child is in one of the left $r$ positions under the vertex above it. It is a right step if the child is in the rightmost position under the vertex above it. A vertex $a$ is a parent of vertex $b$ if there is a path from $a$ to $b$ with each step going away from the root such that it contains exactly one left step. Two vertices are in the same class if they are in the same position under the vertices above them. The root is in the same class as vertices in the $(r+1)$th position.
    
    \item \textbf{Polygonal Divisions.}
    Consider a polygon $A$ with $rmn+2$ sides. Then set $\mathscr{X}_{5,mn}^{\left(r\right)}$ is the set of divisions of $A$ into $(r+2)$-gons containing no new vertices. The elements of a division $\Delta$ are its $(r+2)$-gons. Its levels are determined as follows. Fix once and for all an edge $e$ of $A$. Take polygon $P$ to be the polygon in $\Delta$ containing $e$. Entering $P$ from $e$, there are $(r+1)$ remaining edges, of which we label the $r$ leftmost edges to be left edges and the rightmost edge to be a right edge. As you pass through an edge, you label the remaining edges and continue. A polygon $Q$ is a parent of a polygon $Q'$ if $Q'$ can be reached from $Q$ by passing through a single right edge then an arbitrary number of left edges. Two polygons are in the same class if the edges they shared with their parents are the same number from the left in their parents.

\end{enumerate}

In the course of following Theorem we prove that $\mathscr{X}_{i,mn}^{(r)}$ give the same count.

\begin{thm}\label{3.1}
For each $1 \leqslant i < j \leqslant 5$, there exists a bijection from the objects in $\mathscr{X}_{i,mn}^{\left(r\right)}$ to the objects in $\mathscr{X}_{j,mn}^{\left(r\right)}$ such that there is bijection between elements of corresponding objects and parent/child, class relations are preserved.
\end{thm}

\begin{proof}
The bijections of these sets are similar to the bijections for the normal Catalan numbers. We give brief explanations of the bijections. These bijections make it clear that different notions of admissible labellings agree.
\begin{enumerate}
    \item \textbf{Plane Hypertrees and $r$-Dyck Paths.} From a Dyck Path, we can create a \emph{preorder tree traversal} for a plane hypertree. We start at the root. For every $(1,r)$ step, we create a new hyperedge branching off from the vertex we are at. For every $(1,-1)$ step, we move up one vertex towards the root. We see that will never need to move up from the root vertex since the Dyck Path stays above the x-axis. There are $mn$ steps of the form $(1,r)$, so there will be $mn$ hyperedges in the hypertree. This process is reversible to form a Dyck path given a plane hypertree.
    \item \textbf{$r$-Dyck Paths and Ballot Sequences.} 
To create a ballot sequence from a r-Dyck path, we replace steps of the form $(1,r)$ by $r$ and $(1,-1)$ by $-1$ in the dyck word. This process is easily reversible.

\item \textbf{$(r+1)$-ary Trees and Polygonal Divisions.}
Let $\Delta$ be a $(r+2)$-gongulation of the polygon $\Pi_{rn+2}$ with a distinguished edge $e$. One can make a $(r+1)$-ary tree $R(\Delta)$  by taking the dual of $\Delta$ as follows. The vertices of $R(\Delta)$ are $(r+2)$-gons in $\Delta$. Two vertices are joined by an edge if and only if they correspond to adjacent $(r+2)$-gons in $\Delta$. The distinguished edge $e$ sits on the boundary of one $(r+2)$-gon  which determines the root of $R(\Delta)$. The $(r+1)$-ary tree can now be specified by the direction of edges. We start with the root first. Let the distinguished edge $e$ sit on the boundary of the $(r+2)$-gon $\Gamma$ and the $(r+2)$-gons directly adjacent to $\Gamma$ be $\Gamma_1,\Gamma_2,\ldots,\Gamma_i$ and the vertices inside $\Gamma_i$ be denote $\gamma_i$. The edges of the $(r+2)$-gon containing $e$ when traversed in anticlockwise order starting at the vertex lying on $e$ be labelled $\{1,2,\ldots,r+1\}$ as per their occurrence in increasing order where we don't give the edge $e$ any label. If the edge $\gamma - \gamma_i$ crosses the edge labelled $j$ we give it direction $j$. We repeat the procedure replacing $j$ by $e$. This gives a recursive specification of an $(r+1)$-ary  tree $R(\Delta)$ from $\Delta$ and is seen to be easily reversible. Figure \ref{fig:fuss4} illustrates this process.
\item \textbf{Plane Hypertrees and $(r+1)$-ary Trees.} Let $T$ be a $(r+1)$ uniform hypertree on $n$ hyperedges. We will construct an incomplete $(r+1)$-ary tree $R(T)$ on $n$ vertices. The root hyperedge corresponds to a vertex in $R(T)$. Label the vertices of root hyperedge from left to right $\{v_1,v_2,\ldots,v_{r+1}\}$. If there is a hyperedge hanging from $v_i$ of the root vertex then there is an edge from the root vertex of $R(T)$ to a vertex in direction $i$. We similarly label the hyperedge hanging from $v_i$ from left to right $\{v_1,v_2,\ldots,v_{r+1}\}$ and repeat the procedure. Figure \ref{fig:fuss3} illustrates the process . This process is equivalent to the bijection given by  Bruijn and Morselt \cite{de1967note} for regular graphs. 
\end{enumerate}

\end{proof}

Next, to prove some combinatorial equivalences, we first define a few sets. Let
\[\mathscr{A} = \mathscr{A}_{mn}^{\left(r\right)} = \{(H,l) \mid  H \in \mathscr{X}_{mn}^{\left(r\right)}\text{ and } l \text{ is an admissible } m \text{-labeling}\}.\]
Also, let $\mathscr{F} = \mathscr{F}^{(r,m)}$ be set of tours on the associated bipartite graph of a hyperedge with $r+1$ vertices
\[\{u \mid u \text{ an } a_T^{\left(r,m\right)}\text{-tour} \}\] 
modulo the equivalence relation $u = u'$ if there is an automorphism of the vertices of the hyperedge that takes $u$ to $u'$. Note that $|\mathscr{F}^{r,m}| = FC_1^{\left(r,m\right)}$.

Lastly, let $\mathscr{C}$ be the set of pairs
\[\{(T,w) \mid T \in \mathcal{H}_n^{\left(r\right)} \text{, } w \text{ an } a_T^{\left(r,m\right)}\text{-tour}\}.\]

\begin{thm}\label{3.2}
For all $r,m \geqslant 1$ and $n \geqslant 0$,
\[FC_n^{\left(r,m\right)} = (FC_1^{\left(r,m\right)})^n \left|\mathscr{A}_{mn}^{\left(r\right)}\right|.\]
\end{thm}
\begin{proof}

Let $\mathscr{C}$ be the set of pairs
\[\{(\Tilde{T},w) \mid T \in \mathcal{H}_n^{\left(r\right)} \text{, } w \text{ an } a_T^{\left(r,m\right)}\text{-tour}\}\]

modulo the equivalence relation $(T,w) = (T',w')$ if $T = T'$ and there is an automorphism of $T$ taking $w$ to $w'$. We prove our claim by making a bijection from $\mathscr{C}$ to $\mathscr{A} \times \mathscr{F}^n$.

Next, we define two maps, $\alpha: \mathscr{C} \rightarrow \mathscr{A} \times \mathscr{F}^n$ and $\beta: \mathscr{A} \times \mathscr{F}^n \rightarrow \mathscr{C}$. We define $\alpha$ first. Given $(\Tilde{T},w) \in \mathscr{C}$, let $v$ be the vertex where $w$ begins and ends. We formulate the plane hypertree and labeling of $\alpha(T,w)$ recursively. We start with the root of the plane hypertree. The number of hyperedges connected to the root is the number of times $v \rightarrow e$ for some white vertex $e$ occurs in $w$. From left to right, the $i$th hyperedge is given the label of the white vertex in the $i$th occurrence of $v \rightarrow e$. Consider $v'$, the $i$th vertex, $2 \leqslant i \leqslant r+1$, of hyperedge $e'$ in the plane hypertree. Say this hyperedge is the $j$th hyperedge from left to right with its labelling. Let $e$ be the label of $e'$ and $v_i$ be the $i$th vertex of $e$. The number of hyperedges branching off $e$ from vertex $v'$ is the number of occurrences of $v_i \rightarrow f$ heading away from $v$ in $w$ after the $j$th occurrence of $e \rightarrow v_i$ and before the $(j+1)$th occurrence of $e \rightarrow v_i$ if $j < m$. From left to right, the $k$th hyperedge is given the label of $f$ in the $k$th occurrence of $v_i \rightarrow f$. This recursively defines a hypertree and and labelling $(H, l)$

The labelling is an admissible $m$-labelling because in $w$ each hyperedge in $T$ is entered from its parent hyperedge exactly $m$ times. Two hyperedges $e'_1, e'_2$ which have the same labelling are in the same class since they both branch from the $i$th vertex of their parents. Lastly, their parents are both labelled with $e$.

Now we formulate the tuple $U = (u_1,u_2,\dots,u_n) \in \mathscr{F}^n$. Label the white vertices in $\Tilde{T}$ from $e_1$ to $e_n$. Then, label the black vertices adjacent to each $e_i$ from $v_{i,1}$ to $v_{i,r+1}$ where $v_{i,1}$ is closest to $v$.  Let $e', v'_1, v'_2, \dots, v'_{r+1}$ be the vertices of the associated bipartite of a hyperedge with $r+1$ vertices. For each $i$, we create $u_i$ using all edge crossings of form $v_{i,j} \rightarrow e_i$ and $e_i \rightarrow_{i,j}$ in $w$ in order. An edge crossing of form $v_{i,j} \rightarrow e_i$ in $w$ corresponds to $v'_j \rightarrow e'$ in $u_i$ and an edge crossing of form $e_i \rightarrow_{i,j}$ in $w$ corresponds to $e \rightarrow v'_j$ in $u_i$.

We claim that $\alpha$ is well-defined. Suppose $\alpha(\Tilde{T},w) = ((H,l),U)$ and $(\Tilde{T},w) \sim (\Tilde{T}',w')$. Then, $\alpha(\Tilde{T}',w') = (H,l')$ where the labels of $l'$ are a permutation of those in $l$. Its easy to see that $\alpha$ is injective.

We now define $\beta$. Let $((H,l),U) \in \mathscr{A} \times \mathscr{F}^n$. Let $H'$ be the graph obtained from $H$ by identifying hyperedges with the same label. Note that $H'$ is a hypertree. The admissibility of the labelling implies on hyperedges is $m:1$, the map on vertices away from the root is $m:1$, and the map on the root is $1:1$. This gives a hypergraph on $rn+1$ vertices and $n$ edges, which makes it indeed a $r$-uniform hypertree. As in the previous function, label the vertices of $\Tilde{H}'$ with $e_i$ and $v_{i,j}$.

Next, we define a walk $w$ on $\Tilde{H}'$. We provide casework for where the walk will continue to.
\begin{itemize}
    \item When $w$ is at the root of $\Tilde{H}'$ for the $p$th time, if $e'$ is the $p$th hyperedge directly under the root in $H$, and $e_i$ is the label of $e'$, the walk travels on the edge $v_{i,1} \rightarrow e_i$.
    \item When $w$ enters a hyperedge $v_{i,1} \rightarrow e_i$ for the $p$th time, and the $p$th occurrence of $v'_1 \rightarrow e'$ in $u_i$ is followed by $e' \rightarrow v'_j$, the walk continues to $e_i \rightarrow v{i,j}$.
    \item When $w$ enters a hyperedge in the form if $v_{i,j} \rightarrow e_i$ for $q$th time between the $p$th and $(p+1)$th (end if $p = m$) occurrence of $v_{i,1} \rightarrow e_i$, and the $q+1$th occurrence of edges of the form $e' \rightarrow v'_j$ between the $p$th and $(p+1)$th occurrence of $v'_1 \rightarrow e'$ in $u_i$ exists and is $e' \rightarrow v'_{j'}$, then the walk continues to $e_i \rightarrow v_{i,j'}$.
    \item When $w$ enters a hyperedge in the form if $v_{i,j} \rightarrow e_i$ for $q$th time between the $p$th and $(p+1)$th (end if $p = m$) occurrence of $v_{i,1} \rightarrow e_i$, and there are exactly $q$ occurrences of edges of the form $e' \rightarrow v'_j$ between the $p$th and $(p+1)$th occurrence of $v'_1 \rightarrow e'$ in $u_i$, then the walk continues to $e_i \rightarrow v_{i,1}$.
    
    \item When $w$ goes $e_i \rightarrow v_{i,j}$ where $v_{i,j}$ is a vertex in no other hyperedge, the walk is continued by $v_{i,j} \rightarrow e_i$.
    
    \item For any hyperedge $e_i$ who has child $e_{i'}$ for which they share vertex $v_{i',1}$, $e_{i'} \rightarrow v_{i',1}$ is always followed by $v_{i',1} \rightarrow e_i$ in the walk.
\end{itemize}

The $m$-admissibility of the labelling makes this a proper tour. 

We claim that $\beta$ is well defined. Suppose $\beta((H,l),U) = (\Tilde{T},w)$ and $U' = (u'_1, u'_2, \dots, u'_n)$ such that $u'_i \sim u_i$ for all $i$. Then, $\beta((H,l),U') = (\Tilde{T},w)$ where the vertices within each hyperedge in $\Tilde{T}$ are permuted.

To complete the proof, it suffices to prove that $\beta \circ \alpha = 1_{\mathscr{C}}$ and $\alpha \circ \beta = 1_{\mathscr{A} \times \mathscr{F}^n}$.

We start with $\beta \circ \alpha = 1_{\mathscr{C}}$. We prove that $(\beta \circ \alpha)(\Tilde{T},w) = (\Tilde{T},w)$ by induction on $n$.
The claim is clearly true for $n=0$.
For the induction step, we assume the claim holds for $n = k-1$, $k > 0$. Let $T$ be a hypertree with $n$ hyperedges and $w$ be a walk on $\Tilde{T}$. Let $e_x$ be a hyperedge on $T$ that shares a vertex with exactly one other hyperedge. Let $T'$ be the tree formed from $T$ by deleting $e_x$ and $w'$ be a walk on $\Tilde{T}'$ that is modified appropriately from $w$. In $\alpha(\Tilde{T}',w') = ((H',l'),U')$, there are vertices $v_1, \dots, v_m$ in $H'$ that map to $v_{x,1}$ in $T'$. Then, $((H,l),U)$ is obtained from attaching $m$ hyperedges labelled with $e_x$ to these $v_i$ according to $w$. Also, $u_x$ is determined by $w$. By the inductive hypothesis, $\beta((H',l'),U') = (\Tilde{T}',w')$. Then, $\beta((H,l),U)$ is obtained by collapsing those $m$ hyperedges we attached and attaching it to $v_{x,1}$ in $T$. The positions of the hyperedges in $H$ and $u_x$ determine the correct way to add to $w'$ to get $w$. By induction, the claim is true.\newline
The proof for $\alpha \circ \beta = 1_{\mathscr{A} \times \mathscr{F}^n}$ is done in a similar manner. Thus, we have a proven a bijection between $\mathscr{C}$ and $\mathscr{A} \times \mathscr{F}^n$. Therefore, we have 
\[FC_n^{\left(r,m\right)} = \left(FC_1^{\left(r,m\right)}\right)^n |\mathscr{A}_{mn}^{\left(r\right)}|\]
as desired.
\end{proof}

\section{Generating Function}\label{Genfunc}

\subsection{Calculating $FC_{1}^{(r,m)}$}

Recall the definition of the Hypergraph Fuss-Catalan numbers given in Section~\ref{hypergraph fuss}. Specifically, we must consider the set $\mathcal{H}_1^{(r+1)}$, which consists of 1 hypertree $T$; the hypertree on $r+1$ vertices with 1 edge, so all vertices are contained in the same edge. The associated $\Tilde{T}$ is then one white vertex with $r+1$ black vertices each connected to it by an edge (the star graph $S_{r+1}$ with the center vertex colored white and all leaves colored black). 

Next, we count $a_T^{(r,m)}(v)$ for each black vertex on this graph. Consider labelling the vertices $1$ through $r+1$ in some order (the particular order does not matter), and say we want to find $a_T^{(r,m)}(i)$ for some $1 \leqslant i \leqslant r+1$. We have the following.

\begin{prop}
There is a bijection between the number of valid walks on $\Tilde{T}$ beginning at vertex $i$ ($a_T^{(r,m)}(i)$) and the number of sequences of $m$ $1$s, $m$ $2$s, and so on up to $m$ $(r+1)$s such that no two adjacent elements of the sequence are equal, the last element of the sequence is $i$, and the first element of the sequence is not $i$ which we denote by $CS^{(r+1,m)}$. 
\end{prop}
\begin{proof}
Consider the middle of a walk at a black vertex $j \in \Tilde{T}$. To continue the walk, one must travel to the central white vertex, then to any vertex $k \neq j \in \Tilde{T}$. The walk must begin at vertex $i$ and end at vertex $i$ to be valid, and, not including the beginning vertex, must hit every black vertex of $\Tilde{T}$ exactly $m$ times. This is because an edge is only traversed when travelling to or from the black vertex it contains. Each edge is traversed twice each time its corresponding vertex is crossed, so each black vertex of $\Tilde{T}$ must be passed exactly $m$ times for the walk to be valid, not including the beginning. 

Thus, a valid walk corresponds to a sequence of numbers that consists of $i$ followed by $m$ each of 1, 2, and up to $(r+1)$, such that no two adjacent elements are equal and the last element is $i$. Since the walk starts at $i$, the second vertex hit cannot be $i$. Disregarding the beginning $i$ yields that each valid walk on $\Tilde{T}$ maps to a sequence of the desired form. The process is reversible, which shows the bijection. 
\end{proof}

Note that words with adjacent letters unequal are well-studied, with a closed form given in \cite{eifler1971sequences}. These are called as \emph{Smirnov words} which are formally defined as follows.
 
\begin{definition}
Given a finite set $W$ of letters, $w=w_1w_2\ldots w_n$ such that $w_i \in W$ is called a \emph{Smirnov word} if $w_{i} \neq w_{i+1}$ for $1 \leqslant i \leqslant n-1$
\end{definition}

Since we are interested in $\sum_{i=1}^{r+1} a_T^{(r,m)}(i)$, we are interested in the number of sequences of $m$ 1s, $m$ 2s, and so on up to $m$ $(r+1)$s such that no two adjacent elements of the sequence are equal and the last element of the sequence is not equal to the first. As stated in \cite{MR4101322}, these are known as $\emph{circular Smirnov words}$. In general, these are cumbersome to enumerate, although there is a known generating function as given in \cite{stanley1995symmetric} .  We recall the generating function of Smirnov words over alphabet $\{x_1,x_2,\ldots,x_n\}$ is \[S = \frac{1}{1-\sum_{i=1}^{n} \frac{x_i}{1+x_i}}\] which can be found in \cite{goulden2004combinatorial}. Circular Smirnov words can be specified in terms of Smirnov words as proved below.

\begin{thm}
Let $\mathcal{CS}$ be the set of circular Smirinov words over the alphabet $\{x_1,\ldots,x_n\}$. The ordinary generating function $C(x_1,\ldots,x_n)$ of $\mathcal{CS}$ is given by \[ C(x_1,\ldots,x_n) = \sum_{i=1}^{n}x_i + \sum_{i=1}^{n} \frac{x_i(S(t_i)-1)}{1-x_i(S(t_i)-1)},\]
where $t_i = (x_1,\ldots,x_{i-1},x_{i+1},\ldots,x_n)$
and $S(x_1,\ldots,x_n)$ is the generating function for regular Smirnov words.
\end{thm}
\begin{proof}First we find generating function for Smirnov words that start at $x_1$ and do not end at $x_1$. Such a word is encoded by $x_1 + \text{SEQ}(x_1*(S(x_2,\ldots,x_n)-1))$ which translates to
\[\sum_{n=1}^{\infty}(x_1*(S(x_2,\ldots,x_n)-1))^n + x_1 = x_1 + \frac{x_1*(S(x_2,\ldots,x_n)-1)}{1-x_1*(S(x_2,\ldots,x_n)-1)}\]
where $*$ is the operation for concatenation. Summing over all $x_i$ we have the generating function of circular Smirnov words which is
\[\sum_{i=1}^{n}x_i + \sum_{i=1}^{n} \frac{x_i(S(t_i)-1)}{1-x_i(S(t_i)-1)},\]
where $t_i = (x_1,\ldots,x_{i-1},x_{i+1},\ldots,x_n)$.
\end{proof}
In particular note that 
\[CS^{(r+1,m)} = [(x_1x_2\ldots x_{r+1})^m] C(x_1,x_2,\ldots,x_{r+1})\]
The automorphism group on $S_{r+1}$ is known to be the symmetric group on $r+1$ elements, which has order $(r+1)!$. Thus, we have $$FC_1^{(r,m)} = \frac{\sum_{i=1}^{r+1} a_T^{(r,m)}(i)}{|\Gamma(S_{r+1})|} = \frac{CS^{(r+1,m)}}{(r+1)!}.$$ 
Some examples of $FC_{1}^{(r,m)}$ are 
\begin{align*}
    & r=2 : \,1, 1, 4, 22, 134, 866, 5812, 40048, 281374, 2006698, 14482064, 105527060, \ldots \\
    & r=3 : \, 1, 1, 31, 1415, 75843, 4446741, 276154969, 17851418019, 1188572791275,\ldots \\
    & r=4 : \, 1, 1, 293, 140343, 83002866, 55279816356, 39738077935264,, 30129436868588072,\ldots \\
    & r=5 : \,1, 1, 3326, 20167651, 158861646466, 1450728060971387, 14571371516350429184, \ldots \\
    & r=6 : \, 1, 1, 44189, 3980871156, 490294453324924, 72078730629785796608,\\
    & \,\,\,\,\,\,\,\,\,\,\,\,\,\,\,\,\,\,\,\, 11876790400066162977144832, \ldots \\
    & r=7 : \, 1, 1, 673471, 1035707510307, 2292204611710893056, 6235048155225092628938752, \ldots \\
    & r=8 :\, 1, 1, 11588884, 343866839138005, 15459367618357013512192, \\
    & \,\,\,\,\,\,\,\,\,\,\,\,\,\,\,\,\,\,\,\, 879601407931825671736009949184, \ldots
\end{align*}
This form can be slightly improved using the generating function from the above paragraph.




\subsection{The Generating Function}

We first prove two important propositions.


\begin{prop}\label{4.2}
Let $A(x) = \sum_{k=0}^{\infty} a_kx^k$ be an integral formal power series with $a_k > 0$. Let $\mathcal{P}_A$ be the set of all $(r+1)$ uniform hypertrees such that there is no hyperedge dangling from the root vertex except the hyperedge which contains the root and if a hyperedge has $k_i$ children branching from its $i$th vertex, $2 \leqslant i \leqslant r+1$, the hyperedge can be painted with one of $\prod a_{k_i}$ colors. Let $P_{A}$ be the ordinary generating function of $\mathcal{P}_A$. Then, $P_A$ satisfies the functional relation
$$P_A =  x(A(P_{A}))^r.$$
\end{prop}
\begin{proof}
The singular hyperedge which contains the root can be colored in $\prod a_{k_i}$ ways and to the $i^{th}$ vertex of singular hyperedge we attach $a_{k_i}$ hypertrees from $\mathcal{P}_A$. Thus, $$P_A = x\left(\sum_{(k_2,k_3,\dots,k_{r+1})} \prod a_{k_i} P_A^{k_i}\right) = x (A(P_A))^r$$
which is the desired result.
\end{proof}

\begin{prop}\label{4.3}
Let $B(x) = \sum_{k=0}^{\infty} b_kx^k$ be an integral formal power series with $a_k > 0$. Let $\mathcal{P}_{A,B}$ be the set of all $(r+1)$ uniform hypertrees such that if a hyperedge has $k$ children, it can be painted with one of $a_k$ colors and if the root has $k$ children, it can be colored one of $b_k$ colors. Let $P_{A,B}$ be the ordinary generating function of $\mathcal{P}_{A,B}$. Then, $P_{A,B}$ satisfies the functional relation
$$P_{A,B} =  xB(P_{A}).$$
\end{prop}

\begin{proof}
Let the root have $k$ children. Then, the root can be colored one of $b_k$ colors, and there are $P_A^k$ ways to color the children of the root, continuing recursively. This means $$P_{A,B} = x\left(\sum_{i=0}^{\infty} b_i\left(P_A\right)^i\right) = xB(P_A),$$  as desired.
\end{proof}

\begin{definition}
If a hypertree occurs in $P_{A,B}$ we say that $S$ is equipped with $(A,B)$ coloring.
\end{definition}
Let $\lambda(r,g)$ be the dimension of space of degree $g$ homogeneous polynomials in $r$ variables. We have 
$$\lambda\left(r,g\right) = \binom{r-1+g}{r-1}.$$
\noindent

Define $W_m(k)$ to be the number of ways to partition a set of cardinality $k$ into subsets of cardinality $m$. We have $$W_{m}(k) = \dfrac{k!}{(m!)^{k/m}(k/m)!}.$$
Note that $W_{m}(k)$ is defined iff $m \mid k$.
\begin{thm}
    Define
     $$\ell_{m}(x) = \sum_{d \geqslant 0} W_m(dm)\lambda(m,dm)x^d$$
    and
    $$h_m(x) = \sum_{d \geqslant 0} W_m(dm)x^d.$$
    Let $f_{r,m} \in \ZZ[x]$ satisfy the functional equation
    $$f_{r,m}(x) = x (\ell_m(f_{r,m}(x)))^r.$$
    Then the generating function 
    $$F_{(r,m)}(x) =  \sum_{n \geqslant 0} FC_{n}^{(r,m)}x^{n+1}$$
    satisfies $$F_{(r,m)}(x) = x h_{m}(f_{(r,m)}(FC_{1}^{(r,m)}x)).$$
\end{thm}
\begin{proof} Let $P_{l,h}$ be the set of $l,h$ colored plane hypertrees on $n$ hyperedges. We construct a bijection
$$\rho : \mathscr{A} \to P_{l,h}.$$
For each $(H,l) \in \mathscr{A}$, we form colored hypertree in $P_{l,h}$. The hypertree that will be colored is the hypertree $T$ on $n$ edges that corresponds to $(H,l)$ using the bijection in the combinatorial interpretation section.

The color of the root of $T$ represents the partition of its children hyperedges into labellings as in an admissible $m$-labeling. If the root has $dm$ edges, there are $W_m(dm)$ possible partitions of these hyperedges, giving us $W_m(dm)$ possible colors.

The color of a hyperedge $e$ of $T$ represents both ways to place children under the $m$ hyperedges that correspond to $e$ in $H$ as well as their labeling. We see that hyperedges of different classes are independent from each other, so we can look at each of the second to $(r+1)$ the classes separately. Within a the $i$th class, if there are $d_im$ children, there are $W_m(d_im)$ ways to label them and $\lambda(m,d_im)$ ways to distribute these children among the $m$ parent hyperedges. Thus, there are $\prod W_m(d_im) \lambda(m,d_im)$ ways to set and label these children hyperedges.

This gives a function from $\mathscr{A}$  to $P_{l,h}$. This function is invertible as a $(l,h)$ coloring on a hypertree $T$ on $n$ edges gives a unique $(H,l) \in \mathscr{A}$ using the process above backwards. Thus, we have formed a bijection. Then, by Proposition \ref{4.2} and Proposition \ref{4.3}, we have

\[f_{r,m}(x) = x (\ell_m(f_{r,m}(x)))^r, \quad\sum_{n \geqslant 0} |\mathscr{A}_{mn}^{r}|x^{n+1} = x h_{m}(f_{(r,m)}(x)).\]

Finally, since $FC_{n}^{(r,m)} = (FC_{1}^{(r,m)})^n|A|$, we have 

\[\sum_{n \geqslant 0} FC_{n}^{(r,m)}x^{n+1} = x h_{m}(f_{(r,m)}(FC_{1}^{(r,m)}x))\]
which is the desired result. 
\end{proof}

\section{Future Work}

In the future, we would like to find asymptotics for the generating function of these numbers. We would also like to find more combinatorial interpretations of these numbers, as well as explore more generalizations such as $q$-analogues. 

\section{Acknowledgements}

We would like to thank Lee Trent for her support in mentoring us throughout the project. We would also like to thank Professor Paul Gunnells for proposing the problem and providing us with helpful advice throughout the process. Finally, we would like to thank the Clay Mathematics Institute and PROMYS program for allowing us to undertake this project.

\section{Appendix}
We include data on $FC_{1}^{(r,m)}$ and $\mathscr{A}_{mn}^{(r)}$.
\subsection{$FC_{1}^{\left(2,m\right)}$ ; $0 \leqslant m \leqslant 22$}
\begin{align*}
& 1, 1, 4, 22, 134, 866, 5812, 40048, 281374, 2006698, 14482064, 105527060, 775113440,\\
& 5731756720, 42628923040, 318621793472, 2391808860446,18023208400634,\\ 
& 1033449449559724, 7858699302115444, 59906766929537120,457685157123172672.
\end{align*}
\subsection{$FC_{1}^{\left(3,m\right)}$ ; $0 \leqslant m \leqslant 15$}
\begin{align*}
&1, 1, 31, 1415, 75843, 4446741, 276154969, 17851418019, 1188572791275, 80953196003777,\\
&5613704715433131, 395005886411621632, 28132373164175540224, 2024078159788958023680,\\
&146898874444939943477248.
\end{align*}
\subsection{$FC_{1}^{\left(4,m\right)}$ ; $0 \leqslant m \leqslant 10$}

\begin{align*}
&1, 1, 293, 140343, 83002866, 55279816356, 39738077935264, 30129436868588072,\\
&23760203412845559808,19312059423860889485312, 16076955055099988982890496.
\end{align*}

\subsection{$FC_{1}^{\left(5,m\right)}$ ; $0 \leqslant m \leqslant 10$}
\begin{align*}
&1, 1, 3326, 20167651, 158861646466, 1450728060971387, 14571371516350429184, \\
&156418475586202988707840,1763546149118396438551724032, \\
&249071711627865231410775840362856448.
\end{align*}

\subsection{$FC_{1}^{\left(7,m\right)}$ ; $0 \leqslant m \leqslant 8$}
\begin{align*}
&1, 1, 673471, 1035707510307, 2292204611710893056, 6235048155225092628938752,\\
&19372051918038658908241101062144, 66048441479612871465789854936547196928.
\end{align*}
\subsection{$FC_{1}^{\left(8,m\right)}$ ; $0 \leqslant m \leqslant 7$}
\begin{align*}
    &1, 1, 11588884, 343866839138005, 15459367618357013512192,\\
    &879601407931825671736009949184, 58256941603805586085506513167594815488.
\end{align*}
\subsection{$\mathscr{A}_{2,n}^{(2)}$ ; $0 \leqslant n \leqslant 15$}
\begin{align*}
&1, 1, 9, 126, 2151, 41175, 850176, 18542034, 421860879, 9934669359,\\
&240959223765, 6000574953384, 153165781146996, 4005089138936340, 107341301939872140.
\end{align*}
\subsection{$\mathscr{A}_{3,n}^{(2)}$ ; $0 \leqslant n \leqslant 15$}
\begin{align*}
     &1, 1, 30, 1740, 141400, 14680200, 1906159200, 313012812000, 66433831920000,\\ &18438349698120000, 6629124634968000000, 3011746300705961280000,\\ &1681721795997004320000000, 1127233507688231983249600000
\end{align*}
\subsection{$\mathscr{A}_{2,n}^{(3)}$ ; $0 \leqslant n \leqslant 15$}
\begin{align*}
    &1, 1, 12, 222, 4956, 122985, 3267324, 91059444, 2629956924, 78098264100, \\
    &2371811147640,73388998683990, 2307381601052628, 73568992192119918
\end{align*}
\subsection{$\mathscr{A}_{3,n}^{(3)}$ ; $0 \leqslant n \leqslant 15$}
\begin{align*}
    &1,1, 40, 2920, 284000, 33507600, 4662841600, 769261248000, 154597443264000,\\ &39233217384400000, 12873892988852800000, 5441807779916104960000,\\ &2889966803748531046400000, 1871989337366944472934400000
\end{align*}

\nocite{*}
\bibliographystyle{plain}
\bibliography{ref.bib}

\end{document}